\documentclass{amsart}

\usepackage{amsmath,amssymb,amsthm}

\usepackage{mathtools}
\mathtoolsset{showonlyrefs}

\usepackage{hyperref}
\usepackage{enumerate}

\usepackage{graphicx,tikz}
\newtheorem{theorem}{Theorem}

\newtheorem*{lemma}{Lemma}
\newtheorem{lem}{Lemma}

\theoremstyle{definition}

\theoremstyle{remark}

\newcommand{\C}{\mathbb{C}}
\renewcommand{\H}{\mathbb{H}}
\newcommand{\R}{\mathbb{R}}
\newcommand{\Z}{\mathbb{Z}}

\renewcommand{\L}{\Lambda}

\begin{document}

\title[An Extremal Property of the Hexagonal Lattice]{An Extremal Property of the Hexagonal Lattice}


\author[M.~Faulhuber]{Markus Faulhuber}
\address[M.F.]{Faculty of Mathematics, University of Vienna, 1090 Vienna, Austria}
\email{markus.faulhuber@univie.ac.at}

\author[S.~Steinerberger]{Stefan Steinerberger}
\address[S.S.]{Department of Mathematics, Yale University, New Haven, CT 06511, USA}
\email{stefan.steinerberger@yale.edu}

\thanks{M.F.~is supported by the Erwin Schr\"odinger Program of the Austrian Science Fund (FWF): J4100-N32. S.S.~is supported by the NSF (DMS-1763179) and the Alfred P.~Sloan Foundation.}

\begin{abstract} 
We describe an extremal property of the hexagonal lattice $\Lambda \subset \mathbb{R}^2$. Let $p$ denote the circumcenter of its fundamental triangle (a so-called deep hole) and let $A_r$ denote the set of lattice points that are at distance $r$ from $p$
\begin{equation}
	A_r = \left\{ \lambda \in \Lambda: \| \lambda - p \| = r\right\}.
\end{equation}
If $\Gamma$ is a small perturbation of $\Lambda$ in the space of lattices with fixed density and $C_r$ denotes the set of points in $A_r$ shifted to the new lattice, then
\begin{equation}
	\sum_{\mu \in C_r}{ \| p - \mu\|} - \sum_{\lambda \in A_r}{ \| p - \lambda\|} \gtrsim r \, |A_r| \, d(\Lambda, \Gamma)^2,
\end{equation}
where $d(\Lambda, \Gamma)$ denotes the distance between the lattices: the hexagonal lattice has the property that ``far away points are closer than they are for nearby lattices". 
This has implications in the calculus of variations: assume
\begin{equation}
	g_\Gamma(z) = \sum_{\gamma \in \Gamma} f( \|z - \gamma \|) \quad \textnormal{ satisfies } \quad \min_{z \in \R^2} g_\L(z) = g_\L(p).
\end{equation}
For a certain class of compactly supported functions $f$, the hexagonal lattice $\Lambda$ is then a strict local maximizer of
\begin{equation}
	\max_{\Gamma} \min_{z \in \R^2} \sum_{\gamma \in \Gamma}{f( \|z - \gamma\| )},
\end{equation}
where the maximum runs over all lattices of fixed density. 
\end{abstract}

\maketitle
\vspace{-8pt}
\section{Introduction and Results}
\subsection{Introduction.}
The purpose of this paper is to discuss an extremal property of the hexagonal lattice (sometimes called triangular lattice). The hexagonal lattice appears naturally, sometimes also unexpectedly, as a fundamental object in a wide variety of mathematical problems. In general, the variety of problems, where (expected) solutions involve special lattice configurations, includes e.g. geometric function theory \cite{Bae94}, energy minimization problems \cite{bet, bl, cohn, radin, theil}, packing problems \cite{Via24, con, hales, rogers, steini, Via8}, varying metrics on manifolds \cite{Osgood88} or optimal sampling strategies \cite{strohmer}; we aim to describe a new property of the hexagonal lattice and then explain how that property can be used in the calculus of variations. The novelty in this paper is that we study the behavior of a class of lattice-periodic functions evaluated at deep holes (the circumcenter of a fundamental triangle) of a lattice and not at lattice points.\\

Throughout this work, we use $\Gamma$ to denote a general lattice in $\R^2$ and by $\Lambda$ we denote the standard hexagonal lattice in $\R^2$, normalized to density 1. In general, the area of the fundamental cell of the lattice is called the volume of the lattice and its inverse is called the density. A lattice $\Gamma$ of density 1, can be represented by a matrix $S \in \mbox{SL}(2, \R)$. The matrix $S$ has column vectors $v_1$ and $v_2$, i.e.~$S = (v_1, v_2)$, and the lattice $\Gamma$ is generated by taking integer linear combinations of the vectors;
\begin{equation}
	\Gamma = S \Z^2 = \left\{k v_1 + lv_2 \, \Big| (k,l) \in \Z^2 \right\}.
\end{equation}
Throughout this work, we will write the vectors as row vectors, but they should always be seen as column vectors. Also, we denote the generating vectors of the hexagonal lattice by
\begin{equation}
	v = \frac{\sqrt{2}}{3^{1/4}} \, (1, \, 0)
	\qquad \textnormal{ and } \qquad
	w = \frac{1}{3^{1/4} \sqrt{2}} \, (1,\, \sqrt{3}).
\end{equation}
The circumcenter $p$ of the fundamental triangle of the hexagonal lattice is
\begin{equation}
	p = \frac{1}{3} \, \left( v + w \right) = \frac{1}{3^{1/4} \sqrt{2}} \,
	\left( 1, \, \frac{1}{\sqrt{3}} \right).
\end{equation}

\begin{figure}[ht]
	\includegraphics[width=.7\textwidth]{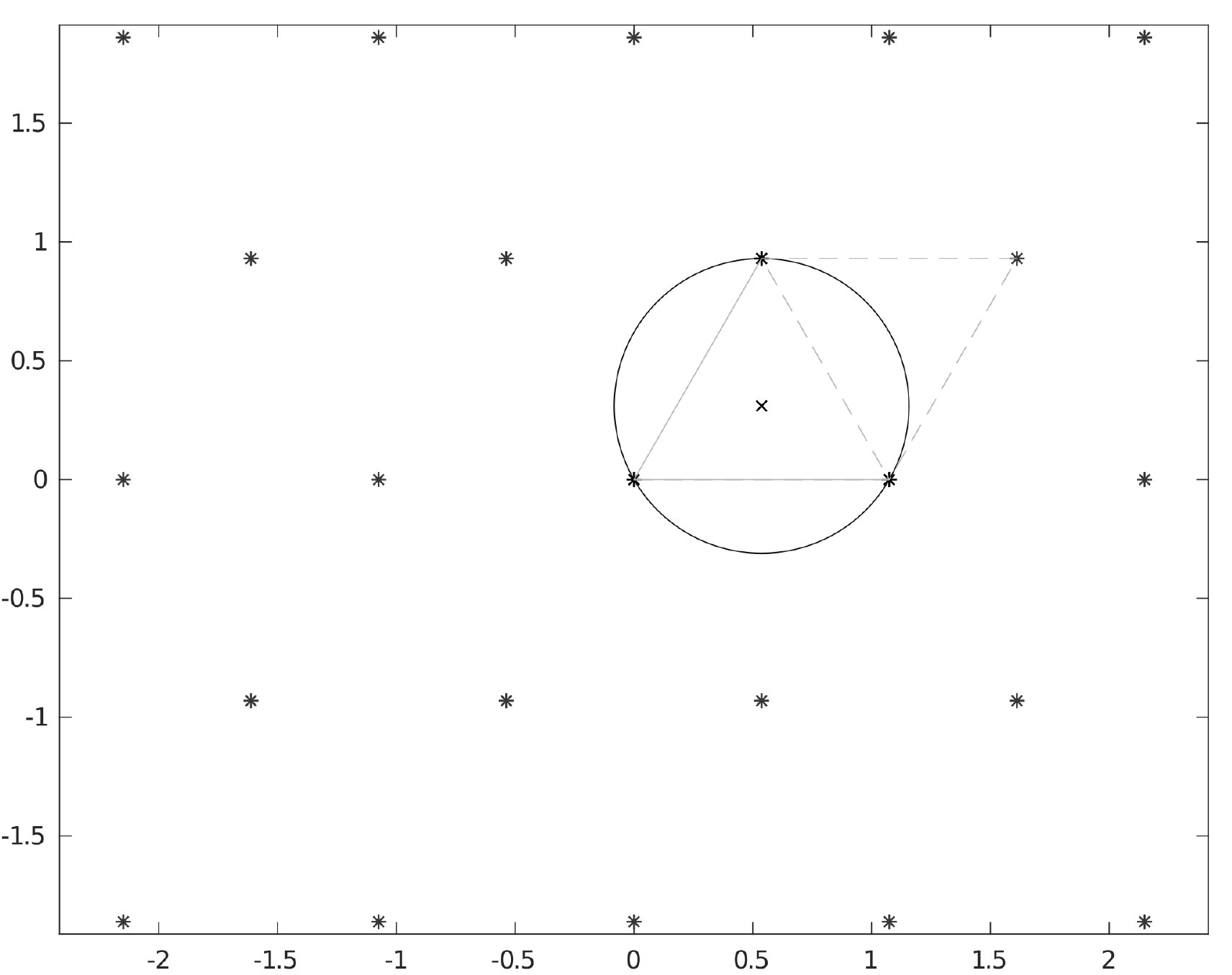}
	\caption{The hexagonal lattice of density 1. The solid lines are the generating vectors $v$ and $w$. The fundamental cell (solid and dashed lines) is the parallelogram spanned by $v$ and $w$, its area equals 1. The $\times$ marks the deep hole $p$.}
\end{figure}

\subsection{The Result.}\label{sec_results}
Let $\Lambda$ be the hexagonal lattice and let $p$ denote a deep hole. By
\begin{equation}
	A_r = \left\{ \lambda \in \Lambda: \| \lambda - p \| = r \right\}
\end{equation}
we denote the set of lattice points at distance $r$ from $p$. This set is always finite but may naturally be empty (depending on $r$).
Our main statement can be summarized as saying that the hexagonal lattice is extremal in the sense that the sum of distances from $A_r$ to the circumcenter of the fundamental cell ($=(\#A_r)r$) is as small as it could possibly be; for any nearby lattice in the space of lattices of the same density, the sum of the distances of the slightly perturbed points to $p$ is larger. We will now make this precise and start by identifying the set of $A_r$ as coordinates with respect to the standard lattice basis vectors $v$ and $w$, which gives rise to the set
\begin{equation}
	B_r = \left\{(k,l) \in \mathbb{Z}^2: \| kv  + lw - p\| = r\right\}.
\end{equation}
If we now slightly perturb the lattice $\Lambda$ in the space of lattices of density 1 (this corresponds to replacing the basis vectors $v$ and $w$ by nearby vectors $v'$ and $w'$ giving rise to a new lattice $\Gamma$), then the new set
\begin{equation}
	C_r = \left\{ k v' + l w': (k,l) \in B_r \right\} \subset \Gamma
\end{equation}
can be understood as a local perturbation of the set $A_r$.
\begin{figure}[ht]
	\includegraphics[width=.7\textwidth]{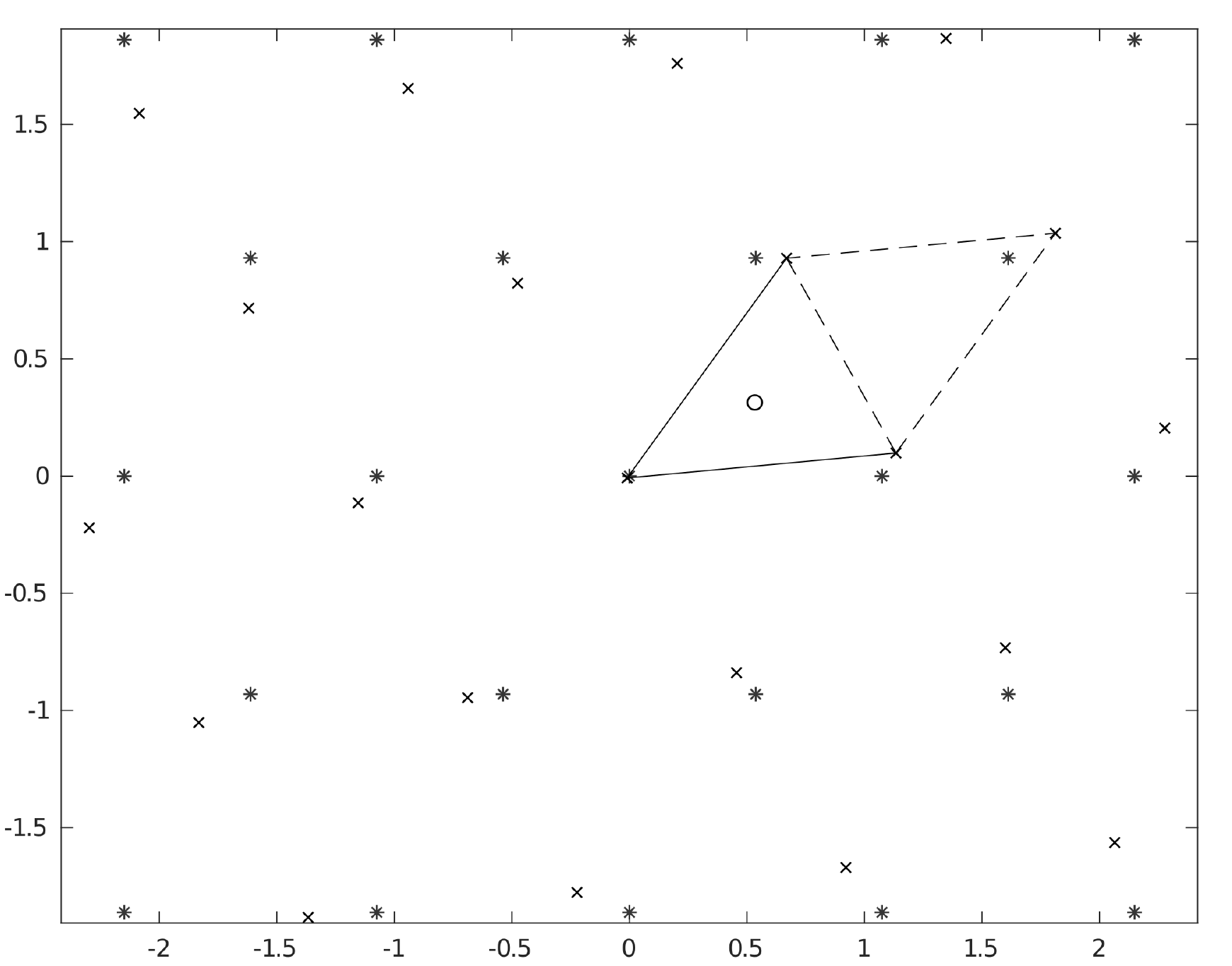}
	\caption{The hexagonal lattice and a perturbed version in comparison. The fundamental cell and a deep hole of the perturbed lattice are shown.}
\end{figure}

We introduce a distance $d(\Gamma_1, \Gamma_2)$ on lattices in the canonical way (say, as the distance in $\mathbb{R}^4$ of the concatenated basis vectors); since it is a finite-dimensional space and we do not specify universal constants, it is not important which distance we choose. More details on the space of lattices are given in $\S$ \ref{sec_latt}.  We can now state our main theorem.

\begin{theorem}\label{thm_distance}
	If $\Gamma$ is sufficiently close $\Lambda$ (depending on $r$), then
	\begin{equation}\label{eq_dist2}
		\sum_{\mu \in C_r} \| p - \mu\| -  \sum_{\lambda \in A_r} \| p - \lambda\| \gtrsim r \, |A_r| \, d(\Lambda, \Gamma)^2.
	\end{equation}
	If $\phi:\mathbb{R}_+ \rightarrow \mathbb{R}$ is any monotonically increasing, convex function, then
	\begin{equation}\label{eq_dist3}
		\sum_{\mu \in C_r} \phi(\| p - \mu\|) -  \sum_{\lambda \in A_r} \phi(\| p - \lambda\|) \gtrsim  r \, \phi'(r) \, |A_r| \, d(\Lambda, \Gamma)^2.
	\end{equation}	
\end{theorem}

Somewhat figuratively put, this means that the hexagonal lattice is one where the circumcenter of the fundamental cell is closer to the points at a certain distance than the furthest point is for any lattice (with the same density) nearby. It could be of interest to understand which lattices in higher dimensions have such a property. One would expect that this is at least loosely connected to sphere packing, which nominates the face centered cubic packing, the hexagonal close packing, $E_8$ and the Leech lattice $\Lambda_{24}$ as natural candidates for further research. We remark that our results do not only hold for density 1 as long as perturbations are done in a way that preserves the density.

\subsection{Calculus of Variations.}
Theorem \ref{thm_distance} has an immediate application in the calculus of variations. Let $\Gamma$ be a lattice in $\mathbb{R}^2$. A classical problem that appears in a variety of settings is to study the function $g_\Gamma:\mathbb{R}^2 \rightarrow \mathbb{R}_+$ where
\begin{equation}
	g_\Gamma (z) = \sum_{\gamma  \in \Gamma}{f(\|z - \gamma\|)},
\end{equation}
with $f:\R_+ \rightarrow \mathbb{R}_+$ satisfying appropriate decay conditions, and ask, for example, for which lattice of fixed density this function $g$ is as `flat' as possible (i.e.~having the smallest global maximum and the largest global minimum). Problems in this direction have been considered e.g. in \cite{Bae94, Bae97, bet, bet1, bet2, bl, markus1, markus2, mont, radin, strohmer}. These problems tend to be fairly difficult: even the special case of $f$ being a Gaussian is not yet fully understood (see \cite{BaeVin98, phd}).\\

We show that, at least for a certain class of functions, the hexagonal lattice has locally (in the space of lattices) the largest minimum. We call a function $f:\R_+ \to \R_+$ admissible if:
\begin{enumerate}[(i)]
\item it is compactly supported
\item it is monotonically decreasing on $(0.5, \infty)$, and
\item it satisfies
	\begin{equation}
		r \, f''(r) \leq -c \, f'(r),
	\end{equation}
	for some universal $c > 0$ in a neighborhood of $\left\{\| \lambda - p\|: \lambda \in \Lambda \right\} \subset \mathbb{R}_+$.\label{cond_3}
\end{enumerate}

Condition (i) combined with condition \eqref{cond_3} shows that such an $f$ must necessarily be discontinuous. The constant $0.5$ in the definition of the domain could be replaced by any $0< R < \sqrt{2}/3^{3/4} \approx 0.62 \ldots$ (the smallest distance between a lattice point of $\Lambda$ and the circumcenter $p$).
\begin{theorem}\label{thm_max}
	Suppose $f$ is admissible and suppose that
	\begin{equation}
		\min_{z \in \R^2} \sum_{\lambda  \in \L} f(\|z - \lambda \|) = \sum_{\lambda \in \L}{f(\|p - \lambda \|)},
	\end{equation}
	then the hexagonal lattice $\Lambda$ is a strict local maximizer of
	\begin{equation}
		\max_{\Gamma} \min_{z \in \R^2} \sum_{\gamma \in \Gamma}{f( \|z - \gamma\| )},
	\end{equation}
	where the maximum runs over lattices $\Gamma$ having density 1.
\end{theorem}
We do not believe that Theorem \ref{thm_max} is necessarily the best kind of result that can be obtained from Theorem \ref{thm_distance} and we believe this to be an interesting avenue for further research. As usual, it is natural to assume that the hexagonal lattice is actually a global and not only a local optimum for many of these problems.

\section{Proofs}

The proof naturally decouples: we first establish a simple reflection symmetry of the hexagonal lattice ($\S$ \ref{sec_lemma}). This symmetry allows us to write the sets $A_r$ as a union of triples of points. The main idea behind the proof of Theorem \ref{thm_distance} is to establish coercivity of the infinitesimal action on triples of points, Theorem \ref{thm_distance} then follows by summation. Theorem \ref{thm_max} follows from the proof of Theorem \ref{thm_distance} together with some basic observations.

\subsection{The Space of Lattices}\label{sec_latt} Before embarking on the proofs, we briefly describe the natural domain for the space of complex lattices which also motivates our parametrization of lattices throughout the rest of the paper. The identification for lattices in $\R^2$ follows from the natural identification of $\C$ with $\R^2$
\begin{equation}
	z = x + i y \in \C \qquad \longleftrightarrow \qquad (x,y) \in \R^2.
\end{equation}

First, we will show how a lattice can be, up to rotation, identified with a point in the complex upper half plane $\H = \{ \tau \in \Z \mid \Im(\tau) > 0\}$. A complex lattice is the integer span of two complex numbers $z_1$ and $z_2$ with the property that $\tfrac{z_1}{z_2} \notin \R$;
\begin{equation}
	\Gamma = \langle z_1, z_2 \rangle_\Z = \{ k z_1 + l z_2 \mid k,l \in \Z, \, z_1, z_2 \in \C, \tfrac{z_1}{z_2} \notin \R \}.
\end{equation}
The area (of the fundamental domain) of the lattice is then given by
\begin{equation}
	\textnormal{area}(\Gamma) = | \Im(\overline{z_1} \, z_2) | = |x_1 y_2 - x_2 y_1|, \qquad z_m = x_m + i y_m, \, m = 1,2.
\end{equation}
We identify lattices which can be obtained from one another by rotation;
\begin{equation}
	\Gamma_1 \equiv \Gamma_2 \qquad \Longleftrightarrow \qquad \Gamma_1 = \zeta \, \Gamma_2, \qquad \zeta \in \C, \, |\zeta| = 1.
\end{equation}

\begin{figure}[ht]
	\includegraphics[width=.7\textwidth]{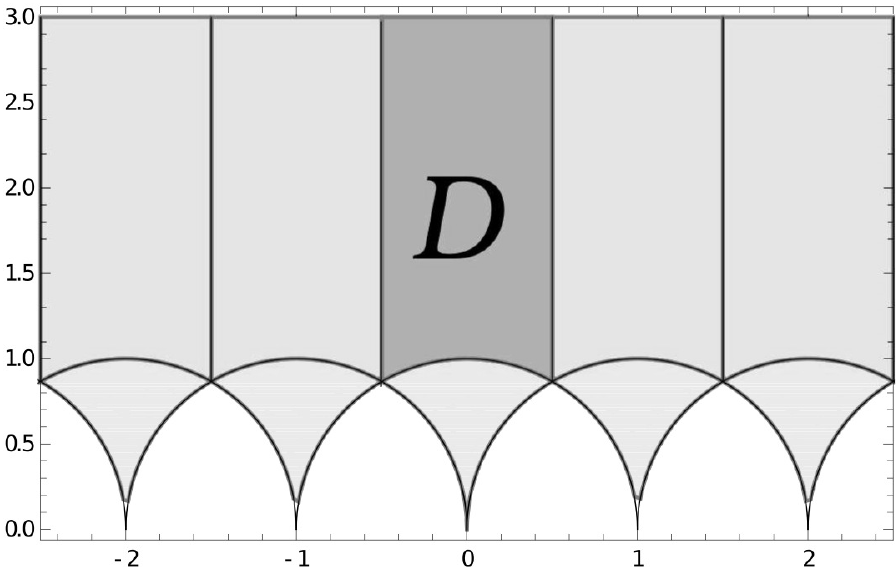}
	\caption{The fundamental domain $D$ and some of its copies derived under the action of the modular group. We note that the picture is fairly incomplete and becomes rather complicated, though structured, close to the real line.}\label{fig_D}
\end{figure}

Therefore, we may assume that $z_1 = 1$. Furthermore, we may assume that $\Im(z_2) > 0$, because if $\Im(z_2) < 0$, then we may simply rotate the lattice such that $z_2 = 1$ and, then, $z_1 = \overline{z_2}$. Hence, a complex number $\tau \in \H$ defines a class of lattices in the complex plane $\C$, which are identified by rotation. By combining all of the previous assumptions, we have for
\begin{equation}
	\Gamma = \langle 1, \tau \rangle, \, \tau \in \H \quad \Longrightarrow \quad \textnormal{area}(\Gamma) = \Im(\tau).
\end{equation}
Hence, to work in the space of lattices with unit area, a point $\tau \in \H$ defines a lattice in the following way;
\begin{equation}
	\Gamma = \Im(\tau)^{-1/2} \langle 1, \tau \rangle.
\end{equation}

As a second step, we are going to show that it is sufficient that $\tau$ is in the so-called fundamental domain of $\H$. The fundamental domain $D$ of $\H$ is given by
\begin{equation}
	D = \{\tau \in \H \mid |\tau| \geq 1, \, |\Re(\tau)| \leq \tfrac{1}{2} \}.
\end{equation}

We need to understand the action of the modular group on $\H$. In general, a matrix $S \in \mbox{SL}(2, \R)$ acts on $\H$ by linear fractional transformations, i.e.~the matrix
\begin{equation}
	S =
	\left(
		\begin{array}{c c}
			a & b\\
			c & d	
		\end{array}
	\right) \in \mbox{SL}(2, \R)
\end{equation}
acts on an element $\tau \in \H$ by the rule
\begin{equation}
	S \circ \tau = \frac{a \tau + b}{c \tau + d}.
\end{equation}
We note that the action of $-I$, where $I$ is the identity matrix, is trivial;
\begin{equation}
	-I \circ \tau = I \circ \tau = \tau.
\end{equation}

Hence, it is actually enough to study the action of the projective special linear group $\mbox{PSL}(2,\R) = \mbox{SL}(2,\R) \slash \{\pm I\}$ on $\H$. This reflects the fact that $\Gamma = -\Gamma$ (after all, $\Gamma$ is an additive group). The discrete subgroup
\begin{equation}
	G = \mbox{SL}(2, \Z) \slash \{\pm I\}
\end{equation}
is called the modular group and it is the image of $\mbox{SL}(2, \Z)$ in $\mbox{PSL}(2,\R)$.
The interpretation of the action of $G$ for the space of lattices in $\R^2$ is the following. The matrix describing a lattice $\Gamma$ is not unique. This is due to the fact that $\Z^2$ (as an additive group) has countably many possible bases. Furthermore, for any $\mathcal{B} \in G$, we have
\begin{equation}
	\mathcal{B} \, \Z^2 = \Z^2,
\end{equation}
which shows that $\mathcal{B}$ is another basis of $\Z^2$. Hence, it follows that if $S$ is a matrix generating a lattice $\Gamma$, then $S \mathcal{B}$ generates the same lattice $\Gamma$.
Therefore, it suffices to focus on lattices generated by $\tau \in D \subset \H$, $\tau = x + i y$. The corresponding lattice in $\R^2$ is generated by the vectors
\begin{equation}
	y^{-1/2} \left(1, \, 0 \right)
	\qquad \textnormal{ and } \qquad
	y^{-1/2} \left(x, \, y \right).
\end{equation}
Any other $\widetilde{\tau} \in \H$ is obtained from $\tau \in D$ by letting the appropriate $\mathcal{B} \in G$ act on it, which corresponds to choosing a different basis for the lattice. For more details on the modular group we refer to \cite[Chap.~7, \S 1]{Serre_1973}.

\subsection{A Lemma.}\label{sec_lemma}
The following lemma states that in the hexagonal lattice $\L$, points at a certain distance from a deep hole naturally occur in triples (see Figure \ref{fig_triples}).
\begin{lemma}\label{lemma}
Let $\Lambda$ denote the standard hexagonal lattice in $\mathbb{R}^2$ spanned by 
\begin{equation}
	v = \left( \frac{\sqrt{2}}{3^{1/4}}, 0 \right) \qquad \mbox{and} \qquad w = \left( \frac{1}{3^{1/4} \sqrt{2}}, \frac{3^{1/4}}{\sqrt{2}}\right).
\end{equation}
By $p = (v+w)/3$ we denote the circumcenter of the fundamental cell and by $R$ the rotation matrix that corresponds to rotation by $2\pi/3$. For any $q \in \mathbb{R}^2$ we have
\begin{equation}
	p+q \in \Lambda \qquad \Longrightarrow \qquad p + Rq \in \Lambda, \; p+R^2q \in \Lambda.
\end{equation}
\end{lemma}

\begin{figure}[ht]
	\includegraphics[width=.9\textwidth]{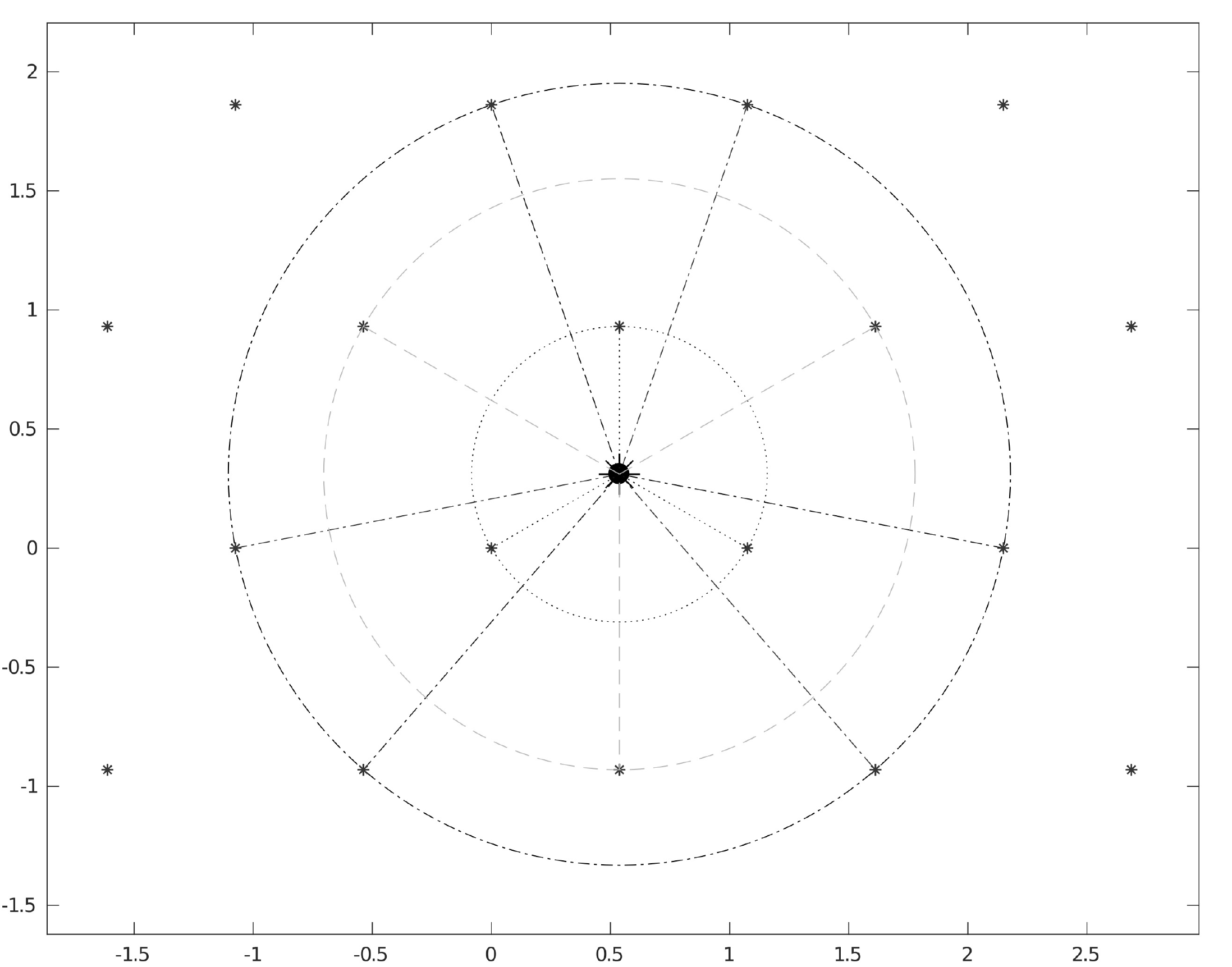}
	\caption{Point-triples of the hexagonal lattice on circle lines at different distances from the deep hole $p$. In general, there is more than one point-triple having a certain distance from $p$.}\label{fig_triples}
\end{figure}

\begin{proof} The argument is purely computational. If $p+q \in \Lambda$, then 
\begin{equation}
	p + q = k v + l w \qquad \mbox{for some}~k,l \in \mathbb{Z}
\end{equation}
and therefore
\begin{equation}
	q = \left( k - \frac{1}{3}\right)v + \left( l - \frac{1}{3}\right) w.
\end{equation}
Moreover, we have
\begin{equation}
	Rv = -v + w, \quad Rw = -v, \quad R^2v = -w, \quad R^2 w = v-w.
\end{equation}
Therefore, we have that
\begin{align*}
	p + Rq &= p +  \left( k - \frac{1}{3}\right)Rv + \left( l - \frac{1}{3}\right)Rw \\
&= (1 - k - l)v + kw \in \Lambda
\end{align*}
as well as
\begin{align*}
 p + R^2q &= p +  \left( k - \frac{1}{3}\right)R^2v + \left( l - \frac{1}{3}\right)R^2w \\
&= lv + (1 - k - l)w \in \Lambda.
\end{align*}
\end{proof}

\subsection{Proof of Theorem \ref{thm_distance}}\label{sec_proof_dist}

We write the basis vectors of a lattice $\Gamma$ with density 1 as
\begin{equation}
	v_1(x,y) = y^{-1/2} \left(1, 0\right) \qquad \mbox{and} \qquad w_1(x,y) = y^{-1/2} \left(x, y\right)
\end{equation}
and note that the hexagonal lattice $\L$ corresponds to the parameter choice $(x,y) = (1/2, \sqrt{3}/2)$ and we write $v$ and $w$ for the corresponding generating vectors.
Recall that
\begin{equation}
	p= \left(\frac{1}{3^{1/4} \sqrt{2}}, \frac{1}{3^{3/4} \sqrt{2}}\right)
\end{equation}
is the deep hole of $\L$.
Suppose now that $\lambda$ is an arbitrary point in $\L$
\begin{equation}
	\lambda = k v + l w \qquad k,l \in \Z .
\end{equation}
We consider its two associated points, obtained by rotation of $2\pi/3$ and $4\pi/3$ around the deep hole $p$, that are given by
\begin{equation}
	\lambda' = p +
	\begin{pmatrix}
		- \frac{1}{2} & -\frac{\sqrt{3}}{2} \\
		\frac{\sqrt{3}}{2} & - \frac{1}{2}
	\end{pmatrix}
 (\lambda-p) \qquad \mbox{and} \qquad
	\lambda'' = p +
	\begin{pmatrix}
		- \frac{1}{2} & \frac{\sqrt{3}}{2} \\
		-\frac{\sqrt{3}}{2} & - \frac{1}{2}
	\end{pmatrix}
	 (\lambda-p).
\end{equation}
The Lemma implies that $\lambda', \lambda'' \in \Lambda$ and yields the representations
\begin{equation}
	\lambda' = (1 - k - l)v + k w \qquad \textnormal{ and } \qquad  \lambda'' = l v + (1 - k - l) w, \qquad k,l \in \Z.
\end{equation}
So far everything is static and carried out for the hexagonal lattice $\Lambda$. We will now investigate the behavior of the triple $(\lambda, \lambda', \lambda'')$ under a perturbation of the lattice, resulting in the triple $(\gamma, \gamma', \gamma'')$ with
\begin{equation}
	\gamma = k v_1 + w_1, \qquad \gamma' = (1-k-l)v_1 + k w_1, \qquad \gamma'' = lv_1 + (1-k-l)w_1.
\end{equation}

We will now first establish the case of squared distances, i.e.~the case \eqref{eq_dist3} for $\phi(r) = r^2$, since it is more computationally tractable and illustrates the main idea very clearly. The case of linear distances \eqref{eq_dist2} is more or less identically but requires more computation, the case of \ convex functions \eqref{eq_dist3} follows from that.
\begin{proof}[Proof for squared distances.]
We will show that, in total, the squared distance of the triple $(\lambda, \lambda', \lambda'')$ to $p$ (which is fixed) increases under a perturbation. For a lattice $\Gamma$, generated by $v_1(x,y)$ and $w_1(x,y)$, we define the function
\begin{equation}
	f (x,y) = \|\gamma - p \|^2 + \|\gamma' - p\|^2 + \|\gamma'' - p\|^2.
\end{equation}
The function can explicitly be written as\begin{align}
	f(x,y) & = \frac{2 \left( \left(x-\tfrac{1}{2}\right)^2+y^2+\tfrac{3}{4}\right)}{y}
	\left(
		\left(k-\tfrac{1}{3}\right)^2 + \left(k-\tfrac{1}{3}\right)\left(l-\tfrac{1}{3}\right) + \left(l-\tfrac{1}{3}\right)^2 - \tfrac{1}{3}
	\right)\\
	& + \frac{1}{3 y}
	\left(
		3 x^2 -  3^{3/4} \sqrt{2} \, x \sqrt{y} + 3 y^2 - 3^{1/4} \sqrt{2} \, y^{3/2} + 2 \sqrt{3} \, y - 3^{3/4} \sqrt{2} \, \sqrt{y}+3
	\right).
\end{align}
For the partial derivatives we get
\begin{align}
	\partial_x f(x,y) & = \frac{2\left(2x-1\right)}{y}
	\left(
		\left(k-\tfrac{1}{3}\right)^2 + \left(k-\tfrac{1}{3}\right)\left(l-\tfrac{1}{3}\right) + \left(l-\tfrac{1}{3}\right)^2 - \tfrac{1}{3}
	\right)\\
	& + \frac{1}{3 y}
	\left(
		6 x - 3^{3/4} \sqrt{2} \, \sqrt{y}
	\right)
\end{align}
and
\begin{align}
	\partial_y f(x,y) & = \frac{2 \left(y^2-\tfrac{3}{4}-\left(x-\tfrac{1}{2}\right)\right)}{y^2}
	\left(
		\left(k-\tfrac{1}{3}\right)^2 + \left(k-\tfrac{1}{3}\right)\left(l-\tfrac{1}{3}\right) + \left(l-\tfrac{1}{3}\right)^2 - \tfrac{1}{3}
	\right)\\
	& + \frac{1}{6 y^2}
	\left(
		-6 x^2 + 3^{3/4} \sqrt{2} \, x \sqrt{y} + 6 y^2 - 3^{1/4} \sqrt{2} \, y^{3/2} + 3^{3/4} \sqrt{2} \, \sqrt{y}-6
   \right).
\end{align}
It is not hard to check that
\begin{equation}
	\partial_x f \big|_{x=1/2, y=\sqrt{3}/2} = 0
	\qquad \textnormal{ and } \qquad
	\partial_y f \big|_{x=1/2, y=\sqrt{3}/2} = 0
\end{equation}
independently of $k, l$. This shows that $\nabla f =0$ in the hexagonal case; $\L$ is a critical point. As a next step we compute the Hessian of $f(x,y)$ at $(x,y) = (1/2, \sqrt{3}/2)$;
\begin{equation}
	H(k,l) = D^2 f\big|_{x=\frac{1}{2}, y=\frac{\sqrt{3}}{2}} = \begin{pmatrix} h_1(k, l) & h_2(k, l) \\ h_2(k, l) & h_3(k, l) \end{pmatrix},
\end{equation}
where $h_1, h_2, h_3$ are rational functions in $k, l$. The results, after simplification, are
\begin{align}
	h_1(k,l) & =
	\frac{		
		8 \left(
			\left(k-\tfrac{1}{3}\right)^2 + \left(k-\tfrac{1}{3}\right)\left(l-\tfrac{1}{3}\right) + \left(l-\tfrac{1}{3}\right)^2 - \tfrac{1}{3}
		\right)
	+4}
	{\sqrt{3}},\\
	h_2(k,l) & = - \frac{2}{3},\\
h_3(k,l) &= h_1(k,l).
\end{align}
We are interested in the behavior of the smaller of the two eigenvalues of $H(k, l)$, which we denote by $\lambda_{\min}(k,l)$. Our goal is to show that it is positive and to deduce a lower bound for it. The characteristic polynomial of the matrix
\begin{equation}
	H(k,l) = \left(
	\begin{array}{c c}
		h_1(k,l) & -\tfrac{2}{3}\\
		-\tfrac{2}{3} & h_1(k,l)
	\end{array}
	\right)
\end{equation}
simply is
\begin{equation}
	P(\lambda) = (h_1(k,l) - \lambda)^2 - \frac{4}{9}.
\end{equation}
From this equation we actually see that
\begin{equation}
	\lambda_{\min}(k,l) = h_1(k,l)-\frac{2}{3}
	\qquad \textnormal{ and } \qquad
	\lambda_{\max}(k,l) = h_1(k,l)+\frac{2}{3}.
\end{equation}
We note that
\begin{equation}
	\left(k-\tfrac{1}{3}\right)^2 + \left(k-\tfrac{1}{3}\right)\left(l-\tfrac{1}{3}\right) + \left(l-\tfrac{1}{3}\right)^2 =
	\left(k-\tfrac{1}{3}, \, l-\tfrac{1}{3}\right)
	\left(
		\begin{array}{c c}
			1 & \frac{1}{2}\\
			\tfrac{1}{2} & 1
		\end{array}
	\right)
	\left(
		\begin{array}{c c}
			k-\tfrac{1}{3}\\
			l-\tfrac{1}{3}
		\end{array}
	\right)
\end{equation}
is a positive definite quadratic form. It follows that
\begin{equation}
	\left(k-\tfrac{1}{3}\right)^2 + \left(k-\tfrac{1}{3}\right)\left(l-\tfrac{1}{3}\right) + \left(l-\tfrac{1}{3}\right)^2 \geq 0,
\end{equation}
with equality only if $(k, l) = (\tfrac{1}{3}, \tfrac{1}{3})$. Hence, we conclude that
\begin{equation}
	h_1(k,l) > \frac{-\tfrac{8}{3} + 4}{\sqrt{3}} = \frac{4}{3 \sqrt{3}}, \qquad \forall (k,l) \in \Z^2.
\end{equation}
Consequently,
\begin{equation}
	\lambda_{\min}(k,l) > \frac{4 - 2 \sqrt{3}}{3 \sqrt{3}} \approx 0.103134 \ldots \; .
\end{equation}

It remains to study the asymptotic behavior of $\lambda_{\min}$. We observe that $h_1(k,l)$ is a positive-definite quadratic form and the off-diagonal terms
are merely a small perturbation. This shows that
\begin{equation}
	\lambda_{\min}(r) \gtrsim  r^2 \qquad \textnormal{ or } \qquad \lambda_{\min}(k,l) \gtrsim k^2+l^2.
\end{equation}
This implies
\begin{equation}
	\sum_{\mu \in C_r} \| p - \mu\|^2 -  \sum_{\lambda \in A_r} \| p - \lambda\|^2 \gtrsim r^2 \, |A_r| \, d(\Lambda, \Gamma)^2.
\end{equation}
\end{proof}
For practical considerations we remark that an explicit computation shows that
\begin{equation}
	\lambda_{\min}(k,l) \geq \frac{4}{\sqrt{3}} - \frac{2}{3} \approx 1.64273 \ldots
\end{equation}
with equality if and only if $(k,l) \in \{(0,0),(1,0),(0,1) \}$. For any norm on the set of lattices of density 1, this allows to determine the exact constant in the result.\\

We also remark that the quadratic form $k^2 + k l + l^2$ is naturally assigned to the hexagonal lattice (see e.g.~\cite[Chap.~2, $\S$ 2.2.]{con}). Geometrically, it describes the ellipse which arises from the circle under the linear transformation which canonically maps the hexagonal lattice to the square lattice, i.e.~the vectors $v$ and $w$ are mapped to the Euclidean unit vectors in $\R^2$. The point $(1/3, 1/3)$ corresponds to the deep hole in the new coordinates and the points $(0,0),(1,0),(0,1)$ are the closest lattice points to the deep hole in these coordinates. This special quadratic form also plays an important role in the theory of cubic theta functions \cite{BorBor91}.

\begin{proof}[Proof of \eqref{eq_dist2}] For the proof of \eqref{eq_dist2} we note that the strategy is the same as in the proof above. However, the computations are a bit more complicated due to the square roots. 
With the notation from above, for a lattice $\Gamma$ we define the function
\begin{equation}
	f (x,y) = \| \gamma - p \| + \| \gamma' - p \| + \| \gamma'' - p\|.
\end{equation}
The function has the explicit form
\begin{align}
	f(x,y) & = \frac{1}{\sqrt{y}}
	\left(
	\sqrt{
	\left(k + l x - \frac{\sqrt{y}}{3^{1/4} \sqrt{2}}\right)^2 + \left(l y - \frac{\sqrt{y}}{3^{3/4} \sqrt{2}}\right)^2
	}
	\right.
	\\
	& + \sqrt{
		\left(1-k-l + k x - \frac{\sqrt{y}}{3^{1/4} \sqrt{2}}\right)^2 + \left(k y - \frac{\sqrt{y}}{3^{3/4} \sqrt{2}}\right)^2
	}
	\\
	& + 
	\left.
	\sqrt{
   	\left((1-k-l) x + l - \frac{\sqrt{y}}{3^{1/4} \sqrt{2}}\right)^2 + \left((1-k-l) y -\frac{\sqrt{y}}{3^{3/4} \sqrt{2}}\right)^2
   	}
   	\right)\\
   	& = \frac{1}{\sqrt{y}} \left( f_1(x,y) + f_2(x,y) + f_3(x,y) \right).
\end{align}
The partial derivatives are not particularly nice, but simplify to
\begin{align}\label{eq_derivative_x}
	\partial_x f(x,y) & =
	\frac{
	l \left(k + l x - \frac{\sqrt{y}}{3^{1/4} \sqrt{2}}\right)
	}
	{
	\sqrt{y} \, f_1(x,y)
	}\\
	& +
	\frac{
	k \left(1-k-l + k x - \frac{\sqrt{y}}{\sqrt{2} \sqrt[4]{3}}\right)
	}
	{
	\sqrt{y} \, f_2(x,y)
	}\\
	& + 
	\frac{
	(1-k-l) \left((1-k-l) x + l - \frac{\sqrt{y}}{\sqrt{2} \sqrt[4]{3}}\right)
	}
	{
	\sqrt{y} \, f_3(x,y)
	}
\end{align}
and
\begin{align}\label{eq_derivative_y}
	\partial_y f(x,y) & =
	\frac{
	2 \left(
		-\frac{k}{2 y}-\frac{l x}{2 y}
	\right)
	\left(
		k + l x - \frac{\sqrt{y}}{3^{1/4} \sqrt{2}}
	\right)
	+ l
	\left(
		l y - \frac{\sqrt{y}}{3^{3/4} \sqrt{2}}
	\right)
	}
	{
	2 \, f_1(x,y)
	}\\
	& +
	\frac{
	2 \left(
		-\frac{1-k-l}{2 y}-\frac{k x}{2 y}
	\right)
	\left(
		1-k-l + k x - \frac{\sqrt{y}}{3^{1/4} \sqrt{2}}
	\right)
	+ k
	\left(k y - \frac{\sqrt{y}}{3^{3/4} \sqrt{2}}\right)
	}
	{
	2 \, f_2(x,y)
	}\\
	& +
	\frac{
	2 \left(
		-\frac{(1-k-l) x}{2 y} - \frac{l}{2 y}
	\right)
	\left(
		(1-k-l) x + l \sqrt{y} - \frac{\sqrt{y}}{3^{1/4} \sqrt{2}}
	\right)
	}
	{
	2 \, f_3(x,y)
	}\\
	& +
	\frac{(1-k-l)
	\left(
		(1-k-l) y-\frac{\sqrt{y}}{3^{3/4} \sqrt{2}}
	\right)
	}
	{
	2 \, f_3(x,y)
	}.
\end{align}
Next, we note that, due to the the rotation symmetry of $\L$ around $p$,
\begin{equation}
	f_1 \left( \tfrac{1}{2}, \tfrac{\sqrt{3}}{2} \right) =
	f_2 \left( \tfrac{1}{2}, \tfrac{\sqrt{3}}{2} \right) =
	f_3 \left( \tfrac{1}{2}, \tfrac{\sqrt{3}}{2} \right).
\end{equation}
Hence, when evaluating $\partial_x f$ and $\partial_y f$ at $(1/2, \sqrt{3}/2)$, the denominators of all summands are the same. So, we may simply sum the numerators in \eqref{eq_derivative_x} and \eqref{eq_derivative_y}, which then yields $\nabla f = 0$ in the hexagonal case. The Hessian of $f$ at $(1/2, \sqrt{3}/2)$ is given by
\begin{equation}
	H(k,l) = D^2 f(x,y)\big|_{x=\frac{1}{2}, y=\frac{\sqrt{3}}{2}} = \begin{pmatrix} h_1(k, l) & h_2(k, l) \\ h_2(k, l) & h_3(k, l) \end{pmatrix},
\end{equation}
where $h_1, h_2, h_3$ are rational functions of $k,l$, where the numerator is of degree 4 and the denominator is of degree 3. We note that we only need to understand the behavior of this matrix for $k, l \in \mathbb{Z}$ (by construction, there is a singularity at $k=l=1/3$ but these values are far away from integers just as the circumcenter is far away from lattice points). A simple computation shows that
\begin{equation}
	\lambda_{\min}(H(k,l)) \geq \frac{9-\sqrt{21}}{2^{3/2} \, 3^{3/4}} \approx 0.685146 \ldots \; ,  \qquad |k| , |l| \leq 100.
\end{equation}
As for the asymptotic bounds we note that, switching to polar coordinates, an easy computation (since they are rational functions) shows that
\begin{align}
	\lim_{r \rightarrow \infty}{ \frac{1}{r} \, h_1(r \cos{t}, r \sin{t}) } &= \frac{3^{3/4}}{2}\sqrt{2 + \sin{2t}} \\
	\lim_{r \rightarrow \infty}{  \, h_2(r \cos{t}, r \sin{t}) } &= \frac{3^{1/4}(3 \cos{t} + \cos{(3t)} + 4 \sin{(3t)}}{4(2 + \sin{(2t)} )^{3/2} }\\
	\lim_{r \rightarrow \infty}{ \frac{1}{r} \, h_3(r \cos{t}, r \sin{t}) } &= \frac{3^{3/4}}{2}\sqrt{2 + \sin{2t}}.
\end{align}
This shows that we may think of $H(k,l)$ as asymptotically being an identity matrix (with size proportional to the distance, the constant depending on the angle in polar coordinates) plus a small perturbation. This, combined with standard error estimates for rational functions, shows that
\begin{align}
	\lambda_{\min}(H(k,l)) &\geq \left( \min_{0 \leq t \leq 2\pi} \frac{3^{3/4}}{2}\sqrt{2 + \sin{2t}}\right) \sqrt{k^2 + l^2} + o(\sqrt{k^2 + l^2}) \\
	& = \frac{3^{3/4}}{2}\sqrt{k^2 + l^2} + o(\sqrt{k^2 + l^2}) \gtrsim \sqrt{k^2 + l^2}.
\end{align}
\end{proof}

\begin{proof}[Proof of \eqref{eq_dist3}]
	It remains to study the general case
	\begin{equation}
		\sum_{\mu \in C_r} \phi(\| p - \mu\|) -  \sum_{\lambda \in A_r} \phi(\| p - \lambda\|).
	\end{equation}
	For $\lambda \in A_r$ and its corresponding point $\mu \in C_r$, we set
	\begin{equation}
		\| p - \mu\| = \|p - \lambda \| + \varepsilon_{\mu} = r + \varepsilon_{\mu}, \qquad \varepsilon_\mu \in \R.
	\end{equation}
	A Taylor expansion of $\sum_{\gamma \in C_r} \phi(r+ \varepsilon_\mu)$ around $\varepsilon_\mu= 0$ shows that
	\begin{equation}
		\sum_{\mu \in C_r} \phi(\| p - \mu\|) -  \sum_{\lambda \in A_r} \phi(\| p - \lambda\|) = \phi'(r) \sum_{\mu \in C_r} {\varepsilon_{\mu}} + \frac{\phi''(r)}{2} \sum_{\mu \in C_r}{\varepsilon_{\mu}^2} + o(d(\Lambda, \Gamma)^2),
	\end{equation}
	where the error term is allowed to depend on $r$ and $A_r$. As we just established,
	\begin{equation}
		\sum_{\mu \in C_r} {\varepsilon_{\mu}}  \gtrsim r \, |A_r| \, d(\L, \Gamma)^2,
	\end{equation}
	which shows the desired statement.
\end{proof}

We note that the estimate $|\varepsilon_{\mu}| \lesssim r \, d(\L, \Gamma)$ implies that we can extend the argument to functions $\phi$ that are almost convex in the sense that they satisfy
\begin{equation}
	c \, \phi'(r) + r \, \phi''(r) \geq 0,
\end{equation}
where $c$ is a universal constant. A very similar computation yields Theorem 2.

\subsection{Proof of Theorem \ref{thm_max}}
\begin{proof}[Proof of Theorem \ref{thm_max}]
We analyze the contribution coming from triples of points at a given distance from $p$ by splitting the total contributions into
\begin{equation}
	\sum_{\lambda \in \Lambda}{ f(\|p - \lambda\|)} = \sum_{r } \sum_{\lambda \in A_r}{ f(\|p - \lambda\|)},
\end{equation}
where the outer sum runs over $r$ with the property that $A_r$ is not empty. Then we only analyze the behavior of a single sum (consisting of one or more triples). As before, we write
\begin{equation}
	\| \mu - q \| = \| \lambda - p\| + \varepsilon_\mu = r + \varepsilon_\mu, \qquad  \varepsilon_\mu \in \R.
\end{equation}
Now, we employ again the estimate from Theorem \ref{thm_distance};
\begin{equation}\label{eq_estimate_Ar}
	\sum_{\lambda \in A_r} \varepsilon_\mu \gtrsim r \, |A_r| \, d(\Lambda, \Gamma)^2.
\end{equation}
A Taylor expansion of $\sum_{\lambda \in A_r} f(r+\varepsilon_\mu)$ around $\varepsilon_\mu = 0$ shows that
\begin{align}
	\sum_{\mu \in C_r}{ f(\| p - \mu\|)} - \sum_{\lambda \in A_r}{ f(\| p - \lambda \|)} &=  f'(r)\sum_{\lambda \in A_r}{\varepsilon_\mu}\\
	 &+ f''(r) \sum_{\lambda \in A_r}{\frac{\varepsilon_\mu^2}{2}}  + o(d(\Lambda, \Gamma)^2) &.
\end{align}
From the estimate we have from Theorem \ref{thm_distance} and $f$ being monotonically decreasing, we see that
\begin{equation}
	f'(r) \sum_{\lambda \in A_r}{\varepsilon_\mu} \lesssim  f'(r) \, r \, | A_r| \, d(\Lambda, \Gamma)^2
\end{equation}
while the estimate $|\varepsilon_\mu| \lesssim r \, d(\Lambda, \Gamma)$ yields
\begin{equation}
	f''(r) \sum_{\lambda \in A_r}{\frac{\varepsilon_\mu^2}{2}} \lesssim f''(r) \sum_{\lambda \in A_r}{ r^2 \, d(\Lambda, \Gamma)^2} = f''(r) \, r^2 \, | A_r| \, d(\Lambda, \Gamma)^2.
\end{equation}
This shows that if, for some universal constant $c$ depending on all other universal constants, 
\begin{equation}
	r\, f''(r) < - c \, f'(r),
\end{equation}
then the negative contributions dominate. This shows 
\begin{equation}
	\sum_{\mu \in C_r}{ f(\| p - \mu\|)} - \sum_{\lambda \in A_r}{ f(\| p - \lambda \|)} \lesssim f'(r) r |A_r| d(\L, \Gamma)^2
\end{equation}
for all perturbations where $d(\L, \Gamma)$ is sufficiently small depending on $r$. Since $f$ is compactly supported, only a finite number of $r$ are relevant and this establishes the desired result.
\end{proof}

\end{document}